\documentclass[12pt, reqno]{amsproc}%
 
\usepackage{graphicx}%

\usepackage[margin=1in]{geometry}
\usepackage{amsmath,amssymb,amsthm,mathtools}
\usepackage{enumitem}

\newtheorem{thm}{Theorem}[section]
\newtheorem{lem}{Lemma}[section]

\newcommand{\R}{\mathbb{R}}
\newcommand{\abs}[1]{\left\lvert #1\right\rvert}
\newcommand{\norm}[1]{\left\lVert #1\right\rVert}
\newcommand{\grad}{\nabla}

\usepackage[colorlinks,linkcolor=blue,citecolor=blue,pagebackref,hypertexnames=false, breaklinks]{hyperref}

\providecommand{\U}[1]{\protect\rule{.1in}{.1in}}

\theoremstyle{plain}

\newtheorem{rem}{Remark}[section]

\numberwithin{equation}{section}

\usepackage[notcite,notref, notshowkeys]{ }

\renewcommand{\tilde}{\widetilde}

\makeatletter
\@addtoreset{equation}{section}

\newcommand{\beq}[1]{\begin{equation} \label{#1}}
\newcommand{\eeq}{\end{equation}}
\newcommand{\bed}{\begin{displaymath}}
\newcommand{\eed}{\end{displaymath}}
\newcommand{\bea}{$$\begin{array}{ll}}
\newcommand{\eea}{\end{array}$$}

\newcommand{\barray}{\begin{array}{ll}}
\newcommand{\earray}{\end{array}}

\newcommand{\bna}{\begin{eqnarray*}}
\newcommand{\ena}{\end{eqnarray*}}

\title[ Trudinger-Moser inequalities on complete noncompact Riemannian manifolds]{ Optimal  Trudinger-Moser inequalities on complete noncompact Riemannian manifolds:  Revisit of the argument from the local inequalities to global ones}

\author{Jungang Li}
\address{Department of Mathematics\\
University of Science and Technology of China\\
Hefei 230026, ANHUI
 CHINA}
\email{jungangli@ustc.edu.cn}

\author{Guozhen Lu}
\address{Department of Mathematics\\
University of Connecticut\\
Storrs, CT 06269, USA}
\email{guozhen.lu@uconn.edu}

\begin{document}

\begin{abstract}
	
The main purpose of this short note, on the one hand,  is  to rigorize  some part of the proof of Theorem 1.3 in \cite{LiLu-AIM} using the level set argument in a simple way, and on the other hand, to give an alternative argument from local inequalities to global ones by using Gromov's covering lemma. This result holds on Riemannian manifolds with positive injective radius and Ricci lower bound as stated in \cite{LiLu-AIM}. Moreover, we show that the Green functions for the Dirichlet problem on any small geodesic balls (or harmonic coordinate charts) on Riemannian manifolds with additional condition that the sectional curvature is bounded  from above  have uniform estimates independent of the locations of the balls (or the charts). This will give another proof on such manifolds   the critical Trudinger-Moser inequality. This latter result of Green's function estimates is of its independent interest and can be found useful elsewhere.

\end{abstract}

	\maketitle
\section{Introduction}

The classical P\'olya-Szeg\"{o} symmetrization inequality plays a fundamental role in establishing geometric and functional inequalities in the Euclidean space. 
However, this symmetrization principle in general fails in many other non-Euclidean settings such as the Heisenberg group or Riemannian manifolds as well as on higher order Sobolev spaces even in the Euclidean space.

To circumvent this obstacle, 
Lam and Lu 
\cite{LamLu2, LamLu1}   developed a symmetrization-free method to prove  critical global Trudinger-Moser inequalities on the entire Heisenberg group, and   Adams inequalities on the higher order Sobolev spaces on the entire Euclidean spaces. 
Lam, Lu and Tang \cite{LamLuTang} establish the subcritical Trudinger-Moser inequality on the Heisenberg group.
The essence of these techniques is to establish the local inequalities on finite domains and then pass them to global inequalities using an argument of level set of functions under consideration. This kind of argument avoids using the P\'oly-Szeg\"{o} inequality and thus is also effective in proving the concentration-compactness principle in  settings where  such symmtrization principle is absent (\cite{LiLuZhu-CVPDE, LiLuZhu-ANS}). This argument has been extensively used by a number of authors and we only refer the reader to  
see also e.g. \cite{YangQ}, \cite{ChenLuZhu-PLMS}, \cite{LuTang-ANS1, LuTang-ANS2} for  adapations of this circle of ideas.  

This sort of ideas has also been used in our proofs of the critical Trudinger-Moser inequality (Theorem 1.3)  on complete and noncompact Riemannian manifolds \cite{LiLu-AIM}. To be precise, we established 

\begin{thm}\label{maintheorem}
		Let $(M, g)$ be a complete noncompact Riemannian manifold whose Ricci curvature is lower bounded, i.e. $\textit{Ric} \geq \lambda g$ for some constant $\lambda \in \mathbb{R}$. Moreover, we assume that its injectivity radius has a lower bound, i.e. $\textit{inj} (M,g) \geq i > 0$ for some constant $i > 0$. Then there exists a positive constant $\epsilon = \epsilon (\lambda , i , n) =O(\frac{1}{r_H})$ in the reverse order of the harmonic radius $r_H$ such that for every $\tau > \epsilon$, there exists a constant $C = C(n, \tau , M)$ such that 
		
		$$
		\sup_{u \in W^{1,n} (M) , ||u||_{1, \tau} \leq 1} \int_M \phi_n (\alpha_n |u|^{\frac{n}{n-1}}) dV_g \leq C. 
		$$ 
		Moreover, $\alpha_n$ is sharp.
	\end{thm}
In the proof of Theorem \ref{maintheorem}
 needs to be rigorized about the uniformity of the upper bound of the critical Trudinger-Moser inequality in terms of the bounded domains of the level sets involved. 
In this note, we will then simply give another simple proof by only establishing  the uniformity of the bounds of the critical Trudinger-Moser inequality  on geodesic balls within the manifolds  by modifying the ideas developed in the earlier work 
by Lam, Lu and Zhang \cite{LamLuZhang}. In fact, we will achieve this by using a scaling argument and the asymptotic estimates for the supremum of the subcritical Trudinger-Moser functional established in \cite{LamLuZhang}. As a consequence of this, we can prove the uniform boundedness of the critical Trudinger-Moser inequality on the level set of the functions under consideration in \cite{LiLu-AIM}. This thus rigorizes the proof of Theorem 1.1 
(namely, Theorem 1.3 in \cite{LiLu-AIM}) by covering the level set  by the geodesic balls.
    Therefore, the same level argument from local inequalities   
    to the global ones used in \cite{LiLu-AIM} are justified. 
     We make use of  the equivalence between the subcritical and critical Moser-Trudinger inequalties on geodesic balls.  This idea was originally in \cite{LamLuZhang} and already used in \cite{LiLu-AIM} to derive subcritical inequalities from critical inequalities. Meanwhile in this note, we will prove in the opposite direction, i.e. from subcritical inequalities to critical inequalities.

    The second main result of this note is to establish  that    Green functions for the Dirichlet problem on any small geodesic balls on Riemannian manifolds with positive injective radius, Ricci curvature lower bounded and sectional curvature upper bounded  have uniform estimates independent of the locations of the balls. We believe that this result is of its independent interest. More precisely,
    
    \begin{thm}\label{Green}
      Let $(M, g)$ be a complete noncompact Riemannian manifold whose Ricci curvature is lower bounded and sectional curvature is upper bounded, i.e. $\textit{Ric} \geq \lambda g$ and $K \leq a^2$ for some constant $\lambda, a \in \mathbb{R}$. Moreover, we assume that its injectivity radius has a lower bound, i.e. $\textit{inj} (M,g) \geq i > 0$ for some constant $i > 0$.  For any $Q >1$ and  $0 < \alpha < 1$, denote the geodesic ball as $B_\delta (x)$ with $\delta \leq r_H(n , i , \lambda , \alpha)$, where $r_H(n , i , \lambda , \alpha)$ is the harmonic radius. Then the Dirichlet Green's function over $B_{\delta} (x)$ satisfies the following estimate:
      
      $$
        |\nabla_y G(x,y)| \leq \frac{1}{\omega_{n-1}} d(x , y)^{1-n} (1 + C d(x , y)),
      $$
      where $C = C(n , \lambda , Q , \delta)$.
    \end{thm}
    
    As a byproduct, this will allow us to conclude that  the Trudinger-Moser inequality on any level set is uniformly bounded independent of the location of the level set, within an extra sectional curvature assumption by using the uniform estimates of the Green functions.

Since the main purpose of this short note is to rigorize  the proof of Theorem 1.3 in \cite{LiLu-AIM} and in the meantime to give an alternative proof from the level set argument passing from  the local inequalities to global one, we have chosen not to give a comprehensive account of the subject (see e.g. \cite{LiLu-AIM} for relevant  references in the literature).

\section{Some  preliminaries in differential geometry}	

We first recall the Gromov's covering lemma:
	
	\begin{lem}[Lemma 1.6 of Hebey \cite{Hebey}]\label{Gromov covering}
		Let $(M,g)$ be a complete Riemannian manifold with Ricci lower bound $\lambda$, let $\delta > 0$ be given, then there exists a sequence of points $\{  x_j \}$ on $M$ such that for any $r \geq \delta$:
		\begin{enumerate}
			\item the family $\{ B_r (x_j) \}$ is a uniformly locally finite covering of $M$, there exists an upper bound for the overlapping number in terms of $n , \delta, r , \lambda$; 
			\item for any $i \neq j$, $B_{r/2}(x_i) \cap B_{r/2} (x_j) = \emptyset$.
		\end{enumerate}
	\end{lem}
	
	Given $(M,g)$ a smooth Riemannian manifold and for any point $P \in M$, we automatically have the normal coordinate on any geodesic balls $B_\rho (P)$ within $\rho \leq \textit{inj} (M,g)$. This is realized by the inverse exponential map $\exp^{-1} : B_\rho(P) \to T_P M \cong \mathbb{R}^n$ such that $\exp^{-1} (P) = 0$. From now on we will denote the normal coordinate function as $\{ \tilde{x}^i \}_{i = 1}^n$. Then the geodesic polar coordinate reads as $(r , \theta) = (  |\tilde{x}|, \tilde{x} / |\tilde{x}| )$. Consider the following boundary value problem
	
	$$
	  \begin{cases}
	    \Delta x^i (Q) = 0 , \ \ \ \textit{for} \  Q \in B_\rho (P); \\
	    x^i (P) = 0 , \frac{\partial x^i}{\partial \tilde{x}^j} (P) = \delta^i_j; \\
	    x^i (Q)= \tilde{x}^i (Q) \ \ \ \textit{for} \ Q \in \partial B_\rho (P).
	  \end{cases}
	$$
	Its solution $x^i$ is called harmonic coordinate for $B_\rho(P)$. Moreover, lower bounds for both Ricci curvature and injectivity radius guarantee the existence of $C^{0,\alpha}$ harmonic coordinates. This is due to Anderson-Cheeger \cite{AndersonCheeger} (see also Hebey \cite{Hebey}) and to be precise:
		
	\begin{lem}\label{harmonic coordinates}
		Let $\alpha \in (0,1)$, $Q > 1$, $\delta > 0$ and $(M,g)$ satisfy $\textit{Ric} \geq \lambda g$ and $\textit{inj} (M , g)  \geq i > 0$. Then there exists a positive constant $C = C(n , Q , \alpha , \delta , i , \lambda)$ such that  for any $x \in M$, the $C^{0,\alpha}$ harmonic radius $r_H(Q , 0 , \alpha) (x) \geq C$. 
	\end{lem}
	
	Combining Lemma \ref{Gromov covering} and Lemma \ref{harmonic coordinates}, we obtain a sequence of balls $\{ B_\delta (x_j) \}$ satisfies following:
	
	\begin{itemize}
		\item For $Q > 1$ fixed (close to 1), $\delta \leq r_H (Q, 0 , \alpha)$ and $Q^{-1} \delta_{ij} \leq g_{ij} \leq Q \delta_{ij}$ and moreover, $g_{ij}$ satisfies the $C^{0,\alpha}$ estimate:
		$$
		r_H^\alpha \sup_{y \neq z} \frac{|g_{ij} (z) - g_{ij} (y) | }{d (y,z)^\alpha} \leq Q - 1.
		$$
		\item In the overlapping charts $B_\delta(x_i) \cap B_\delta(x_j)$, the transition function satifies
		$$
		  ||y^k \circ (x^l)^{-1}||_{C^{1 , \alpha}} \leq C(n , Q , \lambda , \alpha , \delta , i).
		$$ 
		\item $\{ B_{\delta/2} (x_j) \}$ covers $M$ and when $i \neq j$, $B_{\delta/4}(x_i) \cap B_{\delta / 4} (x_j) = \emptyset$;
		\item There exists $N = N(n , \lambda, \delta)$ such that any $x \in M$ belongs to at most $N$ of $B_{\delta} (x_j)$'s.
	\end{itemize}
	\begin{rem}
	  Lemma \ref{harmonic coordinates} is due to  Anderson-Cheeger \cite{AndersonCheeger}. The $C^{0 , \alpha}$ result was achieved by the existence of $W^{1, p}$- harmonic coordinates and a classical Sobolev embedding for $\alpha = 1 - \frac{n}{p}$. This means in each local chart, the metric tensor satisfies 
	  $$
	    r_H^{1 - \frac{n}{p}} ||\nabla g_{ij}||_{p} \leq Q - 1.
	  $$
	  Similar embedding also appears for the coordinate functions $x^i , i = 1 , \dots , n$. We have $W^{2 , p}$ , and hence $C^{1, \alpha}$ uniform bounds.
	\end{rem}

	\section{A path from subcritical inequalities to critical ones: Proof of Theorem \ref{maintheorem}}

This section gives a proof of Theorem \ref{maintheorem} by establishing uniform critical Moser-Trudinger inequalities on geodesic balls from subcritical ones.  By a covering lemma, we can  establish the global ones either by using the level set argument as done in \cite{LiLu-AIM} or the covering lemma. 
 
	\medskip

 We first recall some notations. 
	
	\paragraph{\bf Notation.}
	Let $(M,g)$ be an $n$-dimensional Riemannian manifold, $n\ge 2$.
	Fix $\tau>0$ and define
	$$
	\norm{u}_{1,\tau}:=\norm{\grad u}_{L^n(M)}+\tau\norm{u}_{L^n(M)}.
	$$
	Let
	$$
	\phi_n(t):=e^t-\sum_{k=0}^{n-2}\frac{t^k}{k!},
	\qquad
	\alpha_n:=n\Bigl(\frac{n\pi^{n/2}}{\Gamma(\frac n2+1)}\Bigr)^{\!\!\frac1{n-1}}
	= n\,\omega_{n-1}^{1/(n-1)}.
	$$
	
	\medskip

	We first quote a result from \cite{LamLuZhang}.
	The following theorem  provides the lower and upper
	bounds asymptotically for the supremum of the subcritical Trudinger-Moser inequality in $\mathbb{R}^n$. 
	
	\begin{thm}
		\label{improvedAT}\textit{Let }$n\geq2$ and  $0\leq\alpha<\alpha_{n}.$ Denote%
		$$
		AT\left(  \alpha,\beta\right)  =\sup_{\left\Vert \nabla u\right\Vert _{n}%
			\leq1}\frac{1}{\left\Vert u\right\Vert _{n}^{n-\beta}}\int_{%
			\mathbb{R}
			^{n}}\phi_{n}\left(  \alpha\left(  1-\frac{\beta}{n}\right)  \left\vert
		u\right\vert ^{\frac{n}{n-1}}\right)  \frac{dx}{\left\vert x\right\vert
			^{\beta}}.
	$$
		Then there exist positive constants $c=c\left(  n,\beta\right)  $ and
		$C=C\left(  n,\beta\right)  $ such that when $\alpha$ is close enough to
		$\alpha_{N}:$
		\begin{equation}
			\frac{c\left(  n,\beta\right)  }{\left(  1-\left(  \frac{\alpha}{\alpha_{n}%
				}\right)  ^{n-1}\right)  ^{\left(  n-\beta\right)  /}}\leq AT\left(
			\alpha,\beta\right)  \leq\frac{C\left(  n,\beta\right)  }{\left(  1-\left(
				\frac{\alpha}{\alpha_{n}}\right)  ^{n-1}\right)  ^{\left(  n-\beta\right)
					/n}}. \label{1.3.1}%
		\end{equation}
		Moreover, the constant $\alpha_{N}$ is sharp in the sense that $AT\left(
		\alpha_{N},\beta\right)  =\infty.$
	\end{thm}

	\medskip
	
	\begin{thm}[Uniform local estimate]\label{thm:uniform}
	Given a complete and noncompact Riemannian manifold $M$. 	Assume there exist $\delta_0>0$ and $\Lambda\ge 1$ such that for every $p\in M$
		there is a harmonic coordinate chart
		$$
		\varphi_p: B_{\delta_0}(p)\longrightarrow \R^n
	$$
		whose image contains a Euclidean ball $B_r(0)$ with $r\in[c_1\delta_0,c_2\delta_0]$,
		and such that on $B_{\delta_0}(p)$
		$$
		\Lambda^{-1}\abs{\xi}^2\leq g^{ij}(x)\xi_i\xi_j\leq \Lambda\abs{\xi}^2,
		\qquad
		\Lambda^{-1}\leq \sqrt{\det(g_{ij}(x))}\leq \Lambda.
		$$
		Then there exists $\delta\in(0,\delta_0]$ and $C<\infty$, depending only on $(n,\tau,\Lambda)$,
		such that for every $u\in W^{1,n}(M)$ with $\norm{u}_{1,\tau}\leq 1$,
		$$
		\sup_{p\in M}\int_{B_{\delta/2}(p)}\phi_n\!\left(\alpha_n\abs{u}^{\frac{n}{n-1}}\right)\,dV_g
		\leq C.
	$$
	\end{thm}
	
	\begin{proof}
		Fix $\delta\in(0,\delta_0]$ (chosen at the end depending only on $(n,\tau,\Lambda)$) and fix $p\in M$.
		Let $\varphi:=\varphi_p$ be the harmonic chart on $B_\delta(p)$.
		Write $B_r:=B_r(0)\subset \varphi(B_\delta(p))$ with $r\sim\delta$ (constants depending only on $\Lambda$), and define
		$$
		\tilde u(x):=u(\varphi^{-1}(x)),\qquad x\in B_r.
	$$
		By the metric and Jacobian comparability in harmonic coordinates, Euclidean and Riemannian
		$L^n$-norms on $B_\delta(p)$ and $B_r$ are comparable up to constants depending only on $\Lambda$.
		
		\medskip
		\noindent\textbf{Step 1: Cutoff (localization).}
		Choose $\eta\in C_c^\infty(B_r)$ such that
		$$
		\eta\equiv 1 \ \text{on }B_{r/2},\qquad 0\leq \eta\leq 1,\qquad
		\abs{\grad\eta}\leq \frac{C}{r}.
		$$
		For a scaling parameter $a>0$ (to be chosen), set
		$$
		v(x):=a\,\tilde u(x)\quad (x\in B_r),\qquad w:=\eta v,
		$$
		and extend $w$ by $0$ to $\R^n$. Then $w\in W^{1,n}(\R^n)$ and $w=v$ on $B_{r/2}$, hence
		\begin{equation}\label{eq:reduce}
			\int_{B_{r/2}}\phi_n\!\left(\alpha_n\abs{v}^{\frac{n}{n-1}}\right)\,dx
			=
			\int_{B_{r/2}}\phi_n\!\left(\alpha_n\abs{w}^{\frac{n}{n-1}}\right)\,dx
			\leq
			\int_{\R^n}\phi_n\!\left(\alpha_n\abs{w}^{\frac{n}{n-1}}\right)\,dx.
		\end{equation}
		
		\medskip
		\noindent\textbf{Step 2: Gradient control for the cutoff.}
		Since $\grad w=\eta\grad v+v\grad\eta$,
	$$
		\norm{\grad w}_{L^n(\R^n)}
		\leq
		C\left(\norm{\grad v}_{L^n(B_r)}+\frac{1}{r}\norm{v}_{L^n(B_r)}\right).
		$$
		Using $v=a\tilde u$, $r\sim\delta$, and harmonic coordinate comparability, we obtain
		\begin{equation}\label{eq:gradw}
			\norm{\grad w}_{L^n(\R^n)}
			\leq
			C_1 a\left(\norm{\grad u}_{L^n(B_\delta(p))}+\frac{1}{\delta}\norm{u}_{L^n(B_\delta(p))}\right)
			\leq
			C_1 a\left(\norm{\grad u}_{L^n(M)}+\frac{1}{\delta}\norm{u}_{L^n(M)}\right).
		\end{equation}
		
		\medskip
		\noindent\textbf{Step 3: Choose $a$ uniformly so that $\norm{\grad w}\leq \theta$.}
		Assume $\norm{u}_{1,\tau}\leq 1$, i.e.
		$$
		\tau\norm{u}_{L^n(M)}+\norm{\grad u}_{L^n(M)}\leq 1.
		$$
		Fix a margin $\theta:=\frac12$ and choose
		\begin{equation}\label{eq:defa}
			a:=\frac{1}{2C_1\left(\norm{\grad u}_{L^n(M)}+\frac{1}{\delta}\norm{u}_{L^n(M)}\right)}.
		\end{equation}
		Then \eqref{eq:gradw} gives
		\begin{equation}\label{eq:gradwtheta}
			\norm{\grad w}_{L^n(\R^n)}\leq \theta=\frac12.
		\end{equation}
		(Here $a$ depends on $u$, but the final bound will be uniform because $a^n\norm{u}_{L^n}^n$ is uniformly controlled by $\norm{u}_{1,\tau}\leq 1$.)
		
		\medskip
		\noindent\textbf{Step 4: Apply the asymptotic estimates for the subcritical Trudinger-Moser inequality at a fixed subcritical parameter.}
		Define
		$$
		z:=\frac{w}{\theta}\qquad\Bigl(\text{so that }\ \norm{\grad z}_{L^n(\R^n)}\leq 1\ \text{by \eqref{eq:gradwtheta}}\Bigr).
		$$
		Then
		$$
		\alpha_n\abs{w}^{\frac{n}{n-1}}
		=
		\alpha_n\theta^{\frac{n}{n-1}}\abs{z}^{\frac{n}{n-1}}
		=: \alpha\,\abs{z}^{\frac{n}{n-1}},
		\qquad
		\alpha:=\alpha_n\theta^{\frac{n}{n-1}}<\alpha_n,
		$$
		where $\alpha$ is a \emph{fixed} subcritical constant depending only on $n$.
		By the definition of $\mathrm{AT}(\alpha,0)$,
		\begin{equation}\label{eq:ATapply}
			\int_{\R^n}\phi_n\!\left(\alpha_n\abs{w}^{\frac{n}{n-1}}\right)\,dx
			=
			\int_{\R^n}\phi_n\!\left(\alpha\abs{z}^{\frac{n}{n-1}}\right)\,dx
			\leq
			\mathrm{AT}(\alpha,0)\,\norm{z}_{L^n(\R^n)}^n
			=
			\mathrm{AT}(\alpha,0)\,\theta^{-n}\,\norm{w}_{L^n(\R^n)}^n.
		\end{equation}
		Since $\alpha$ is fixed and $\alpha<\alpha_n$, $\mathrm{AT}(\alpha,0)$ is a finite constant depending only on $n$.
	 
		\medskip
		\noindent\textbf{Step 5: Uniform control of $\norm{w}_{L^n}$ and conclusion.}
		
		Because $0\leq \eta\leq 1$ and $w=\eta v$ on $B_r$,
		$$
		\norm{w}_{L^n(\R^n)}\leq \norm{v}_{L^n(B_r)}=a\,\norm{\tilde u}_{L^n(B_r)}\leq C_2\,a\,\norm{u}_{L^n(B_\delta(p))}\leq C_2\,a\,\norm{u}_{L^n(M)}.
		$$
		Substituting into \eqref{eq:ATapply} yields
		$$
		\int_{\R^n}\phi_n\!\left(\alpha_n\abs{w}^{\frac{n}{n-1}}\right)\,dx
		\leq
		C_3(n)\,a^n\,\norm{u}_{L^n(M)}^n.
		$$
		Now use the definition \eqref{eq:defa}. 
		
		Since
		$$
		\norm{\grad u}_{L^n(M)}+\frac{1}{\delta}\norm{u}_{L^n(M)}
		\geq
		\frac{1}{\delta}\norm{u}_{L^n(M)},
		$$
		we obtain
		 
		$$
		a
		=\frac{1}{2C_1\left(\norm{\grad u}_{L^n(M)}+\frac{1}{\delta}\norm{u}_{L^n(M)}\right)}
		\leq \frac{\delta}{2C_1}\,\frac{1}{\norm{u}_{L^n(M)}}.
		$$
		Therefore
		$$
		a^n\,\norm{u}_{L^n(M)}^n\leq \left(\frac{\delta}{2C_1}\right)^n,
		$$
		and we conclude
		$$
		\int_{\R^n}\phi_n\!\left(\alpha_n\abs{w}^{\frac{n}{n-1}}\right)\,dx
		\leq C_4(n,\Lambda)\,\delta^n.
		$$
		Combining with \eqref{eq:reduce} gives the Euclidean local bound on $B_{r/2}$.
		Finally, pulling back by $\varphi$ and using Jacobian comparability yields
		$$
		\int_{B_{\delta/2}(p)}\phi_n\!\left(\alpha_n\abs{u}^{\frac{n}{n-1}}\right)\,dV_g
		\leq C(n,\tau,\Lambda),
		$$
		with $C$ independent of $u$ and $p$ (after fixing $\delta$ depending only on $(n,\tau,\Lambda)$).
		This proves the theorem.
	\end{proof}
	
	\medskip
	 
	 Next we need to putting together the local estimates over balls to get the estimates over $M$.

	\begin{lem}[Cutoff bounds]\label{lem:cutoff}
		Let $\{x_j\}$ be as in Lemma~\ref{harmonic coordinates}. For each $j$ choose $\psi_j\in C_c^\infty(B_\delta(x_j))$
		such that $0\leq \psi_j\leq 1$, $\psi_j\equiv 1$ on $B_{\delta/2}(x_j)$, and $\abs{\grad\psi_j}\leq 4/\delta$.
		Then $\abs{\grad(\psi_j^2)}\leq 8\psi_j/\delta$ and for every $u\in W^{1,n}(M)$ and every $\tilde\tau\geq 0$,
		\begin{equation}\label{eq:local-norm}
			\norm{\psi_j^2 u}_{1,\tilde\tau}
			\leq
			\norm{\grad u}_{L^n(M)}+\Bigl(\tilde\tau+\frac{8}{\delta}\Bigr)\norm{u}_{L^n(M)}.
		\end{equation}
		Moreover,
		\begin{equation}\label{eq:sum-overlap}
			\sum_j \int_M \psi_j^{2n}\abs{\grad u}^n\,dV_g \leq N\int_M \abs{\grad u}^n\,dV_g,
			\qquad
			\sum_j \int_M \abs{\grad(\psi_j^2)}^n\abs{u}^n\,dV_g \leq \frac{8^n}{\delta^n}\,N\int_M \abs{u}^n\,dV_g.
		\end{equation}
	\end{lem}

	 \begin{proof}
	 	For \eqref{eq:local-norm},  compute
	 	$$
	 	\grad(\psi_j^2u)=\psi_j^2\grad u + u\,\grad(\psi_j^2),
	 	$$
	 	hence by the triangle inequality in $L^n$,
	 	$$
	 	\norm{\grad(\psi_j^2u)}_{L^n(M)}
	 	\leq \norm{\psi_j^2\grad u}_{L^n(M)}+\norm{u\,\grad(\psi_j^2)}_{L^n(M)}
	 	\leq \norm{\grad u}_{L^n(M)}+\frac{8}{\delta}\norm{\psi_j u}_{L^n(M)}
	 	\leq \norm{\grad u}_{L^n(M)}+\frac{8}{\delta}\norm{u}_{L^n(M)}.
	 	$$
	 	Also $\norm{\psi_j^2u}_{L^n}\leq \norm{u}_{L^n}$, so adding $\tilde\tau\norm{\psi_j^2u}_{L^n}$
	 	gives \eqref{eq:local-norm}. The implication  
	 	$$	\norm{\psi_j^2 u}_{1,\tilde{\tau}}
	 	\leq 1$$ 	 	
	 	 follows immediately from
	 	$\tau\geq \tilde\tau+\frac{8}{\delta}$ and $\norm{u}_{1,\tau}\leq 1$.

	 \end{proof}
	 
	\begin{thm}[Global Moser--Trudinger from the local estimate] \ref{harmonic coordinates}
		Assume the hypotheses of Lemma~ \ref{harmonic coordinates}.
		Fix $\delta\in(0,r_H]$ and assume the following \emph{local} estimate holds:
		
		\smallskip
		
		\noindent\emph{\bf (Local MT at scale $\delta$).} There exist $\tilde\tau\ge 0$ and $C_{\mathrm{loc}}>0$
		(depending only on $n$ and the geometric data) such that for every $p\in M$ and every
		$f\in W^{1,n}(M)$ supported in $B_\delta(p)$ with $\norm{f}_{1,\tilde\tau}\leq 1$,
		\begin{equation}\label{eq:local}
			\int_{B_{\delta/2}(p)} \phi_n\!\left(\alpha_n\abs{f}^{\frac{n}{n-1}}\right)\,dV_g
			\leq C_{\mathrm{loc}}.
		\end{equation}
		
		\smallskip
		Then for every $\tau\ge \tilde\tau+\frac{8}{\delta}$ there exists $C_{\mathrm{glob}}<\infty$
		depending only on $(n,\tau)$ and the geometric data such that for all $u\in W^{1,n}(M)$ with
		$\norm{u}_{1,\tau}\leq 1$,
		\begin{equation}\label{eq:global}
			\int_M \phi_n\!\left(\alpha_n\abs{u}^{\frac{n}{n-1}}\right)\,dV_g \leq C_{\mathrm{glob}}.
		\end{equation}
	\end{thm}
	
	\begin{proof}
		Let $\{x_j\}$ and $\{\psi_j\}$ be as in Lemmas~ \ref{harmonic coordinates} and \ref{lem:cutoff}.
		Since $M=\bigcup_j B_{\delta/2}(x_j)$ and $\psi_j\equiv 1$ on $B_{\delta/2}(x_j)$,
		we have pointwise on $B_{\delta/2}(x_j)$:
		$$
		\abs{u}=\abs{\psi_j^2u}.
		$$
		Using that $\phi_n$ is increasing on $[0,\infty)$, we obtain
		\begin{equation}\label{eq:split}
			\int_M \phi_n\!\left(\alpha_n\abs{u}^{\frac{n}{n-1}}\right)\,dV_g
			\leq
			\sum_j \int_{B_{\delta/2}(x_j)} \phi_n\!\left(\alpha_n\abs{u}^{\frac{n}{n-1}}\right)\,dV_g
			=
			\sum_j \int_{B_{\delta/2}(x_j)} \phi_n\!\left(\alpha_n\abs{\psi_j^2u}^{\frac{n}{n-1}}\right)\,dV_g.
		\end{equation}
		
		\medskip
		\noindent\textbf{Apply the local estimate with correct scaling.}
		Assume $\tau\ge \tilde\tau+\frac{8}{\delta}$ and $\norm{u}_{1,\tau}\leq 1$.
		Then Lemma~\ref{lem:cutoff} gives $\norm{\psi_j^2u}_{1,\tilde\tau}\leq 1$ for all $j$,
		and $\psi_j^2u$ is supported in $B_\delta(x_j)$.
		Thus we may apply \eqref{eq:local} with $p=x_j$ and $f=\psi_j^2u$ to obtain
		$$
		\int_{B_{\delta/2}(x_j)} \phi_n\!\left(\alpha_n\abs{\psi_j^2u}^{\frac{n}{n-1}}\right)\,dV_g
		\leq C_{\mathrm{loc}}.
		$$
		To globalize, we use the standard homogeneous consequence of the local estimate:
		
		\smallskip
		\noindent\emph{Homogeneous local bound.}
		There exists $C'_{\mathrm{loc}}>0$ such that for all $p\in M$ and all $f$ supported in $B_\delta(p)$ and satisfying $\norm{f}_{1,\tilde\tau}\leq 1$,
		\begin{equation}\label{eq:local-homog}
			\int_{B_{\delta/2}(p)} \phi_n\!\left(\alpha_n\abs{f}^{\frac{n}{n-1}}\right)\,dV_g
			\leq C'_{\mathrm{loc}}\;\ \norm{f}_{1,\tilde\tau}^{\,n}.
		\end{equation}
	 
		Applying \eqref{eq:local-homog} to $f=\psi_j^2u$ and summing in \eqref{eq:split} gives
		$$
		\int_M \phi_n\!\left(\alpha_n\abs{u}^{\frac{n}{n-1}}\right)\,dV_g
		\leq
		C'_{\mathrm{loc}}\sum_j \norm{\psi_j^2u}_{1,\tilde\tau}^{\,n}.
		$$
		
		\medskip
		\noindent\textbf{Estimate the sum of localized norms using bounded overlap.}
		Using $$(A+B)^n\leq 2^{n-1}(A^n+B^n),$$ the product rule, and $\psi_j\leq 1$,
		$$
		\norm{\psi_j^2u}_{1,\tilde\tau}^{\,n}
		\leq C(n)\int_M \psi_j^{2n}\abs{\grad u}^n + \abs{\grad(\psi_j^2)}^n\abs{u}^n + \tilde\tau^n \psi_j^{2n}\abs{u}^n\,dV_g.
		$$
		Summing over $j$ and using Lemma~\ref{lem:cutoff}  and $\sum_j\psi_j^{2n}\leq N$) yields
		$$
		\sum_j \norm{\psi_j^2u}_{1,\tilde\tau}^{\,n}
		\leq
		C(n)N\int_M \abs{\grad u}^n\,dV_g
		+
		C(n)N\left(\frac{8^n}{\delta^n}+\tilde\tau^n\right)\int_M \abs{u}^n\,dV_g.
		$$
		If $\norm{u}_{1,\tau}\leq 1$ then $\int_M\abs{\grad u}^n\leq 1$ and $\int_M\abs{u}^n\leq \tau^{-n}$.
		Therefore
		$$
		\int_M \phi_n\!\left(\alpha_n\abs{u}^{\frac{n}{n-1}}\right)\,dV_g
		\leq
		C'_{\mathrm{loc}}\,C(n)N\left(1+\left(\frac{8^n}{\delta^n}+\tilde\tau^n\right)\tau^{-n}\right)
		=:C_{\mathrm{glob}},
		$$
		which is finite and depends only on $(n,\tau)$ and the geometric data.
		This proves \eqref{eq:global}.
	\end{proof}

	\section{Green's function estimate: Proof of Theorem \ref{Green}} 
	
	For $u \in W^{1,n}_0(B_\delta (x_j))$, we recall the pointwise representation formula of $u(x)$ in terms of the Dirichlet Green's function:
	
	$$
	|u(x)| \leq \int_{B_\delta (x_j)} |\nabla G(x , y)| |\nabla u (y)| dV_g(y).
	$$
	The existence of the Green's function comes from a classical paramatrix argument (see e.g. Chapter 4 of Aubin \cite{Aubin}). To be precise, for some positive integer $k > \frac{n}{2}$, 
	
	$$
	G(x, y ) = H(x, y ) + \sum_{l = 1}^k \int_{B_\delta (x_j)} \Gamma_l (x, z) H (z , y) dV_g(z) + F (x, y ),
	$$
	where $H(x , y) = \frac{1}{(n-2) \omega_{n-1} } d(x, y )^{2- n} f(d(x,y))$ is the fundamental solution of Laplacian multiplied by a smooth cutoff function $f$, which is identically equal to 1 when $d(x , y) < \frac{\delta}{2}$ and vanishes when $d(x,y) \geq \delta$. Moreover, $f$ can be chosen such that $|f'| \leq \frac{C}{\delta}, |f''| \leq \frac{C}{\delta^2}$ for some uniform constant $C > 0$. $\Gamma (x ,y ) = \Gamma_1 (x , y) = - \Delta_y H(x , y)$ and we inductively define 
	
	$$
	\Gamma_{l + 1} (x , y ) = \int_{B_\delta (x_j)} \Gamma_l (x , z) \Gamma (z , y) dV_g(z).
	$$
	The error term $F$ satisfies the equation
	
	\begin{equation}\label{error term}
		\begin{cases}
			\Delta_y F(x , y) = \Gamma_{k+1} (x , y) , \\
			\ F(x , y) = 0 \ \ \  \textit{for } y \in \partial B_\delta(x_j).
		\end{cases}
	\end{equation}
	We claim that 
	
	\begin{equation}\label{Green estimate}
		|\nabla_y G(x , y)| \leq  \frac{1}{\omega_{n-1}} d(x , y)^{1-n} (1 + C d(x , y)),
	\end{equation}
	where $C = C(n , \lambda , Q , \delta)$. Indeed, from the definition of $\Gamma_l(x, y )$, their singularity decrease as $l$ grows hence we first look at 
	$$
	\int_{B_\delta(x_j)} \Gamma(x , z) H(z , y) dV_g(z).
	$$
	For the asymptotic behavior near singularity, we need to prove the following lemma:
	
	\begin{lem}\label{Singularity Lemma}
		For any $\alpha , \beta > 0$ such that $\alpha + \beta < n$, there exists a constant $C = C(n , \alpha , \beta  , Q)$ such that 
		
		$$
		\int_{B_\delta(x_j)} d(x , z)^{\alpha - n} d(z , y)^{\beta - n} dV_g(z) \leq C d(x, y )^{\alpha + \beta - n}.
		$$
	\end{lem}
	
	\begin{proof}
		Let $\rho = \frac{d(x , y)}{2}$. Then 
		
		\begin{align}\label{convolution identity}
			&\int_{B_\delta(x_j)} d(x , z)^{\alpha - n} d(z , y)^{\beta - n} dV_g(z) \\
			= &\left(  \int_{B_\rho(x) \cap B_\delta(x_j)} +  \int_{ (B_{3\rho} (y) \setminus B_\rho (x) ) \cap B_\delta(x_j)} +  \int_{B_\delta(x_j) \setminus B_{3\rho} (y)} \right) d(x , z)^{\alpha - n} d(z , y)^{\beta - n} dV_g(z).
		\end{align}
		The first term of the right hand side can be estimated as
		
		\begin{align*}
			&\int_{B_\rho(x) \cap B_\delta(x_j)} d(x , z)^{\alpha - n} d(z , y)^{\beta - n} dV_g(z) \\
			\leq & \rho^{\beta - n} \int_{B_\rho(x) \cap B_\delta(x_j)} d(x , z)^{\alpha -n } dV_g(z), \\ 
			\leq & C(\lambda , \delta, Q) \omega_{n-1} \rho^{\beta - n} \int_0^\rho r^{\alpha - n} r^{n-1} dr, \\
			\leq &C(\lambda, \delta, Q) \omega_{n-1} \rho^{\alpha + \beta - n},
		\end{align*}
		where from the second line to the third line we use the property that in harmonic coordinate, $Q^{-1} \delta_{ij} \leq g_{ij} \leq Q \delta_{ij}$ as bilinear form. Hence integral can be transferred to normal coordinate and further to polar coordinate. Meanwhile it is worth to emphasize that we omit the transition function between harmonic coordinate charts centered at $x_j$ and centered at $x$, since it is a $C^{1, \alpha}$ function and can just be treated as a uniform constant. We will also apply such treatment in future estimates.
		Similarly we can estimate the second term of the right hand side of \eqref{convolution identity}. For the third term, we have
		
		\begin{align}
			& \int_{B_\delta(x_j) \setminus B_{3\rho} (y)} d(x , z)^{\alpha - n} d(z , y)^{\beta - n} dV_g(z) \\
			\leq & C_n \int_{B_\delta(x_j) \setminus B_{3\rho} (y)} d(z , y)^{\alpha + \beta - 2n} dV_g(z), \\
			\leq &  C(n , \lambda, \delta, q)  \int_{3 \rho}^\delta r^{\alpha + \beta - 2n} r^{n-1} dr , \\
			\leq & C_{n , \alpha ,\beta , \lambda, \delta}  \omega_{n-1} ( (3\rho)^{\alpha + \beta - n} - \delta^{\alpha + \beta - n} ), \\
			\leq &C_{n , \alpha ,\beta , \lambda, \delta} \omega_{n-1} \rho^{\alpha + \beta - n}.
		\end{align}
	\end{proof}
	It is easy to verify that in polar coordinates, 
  $$
  \Delta_z H(x , z) = \frac{1}{(n-2) \omega_{n-1} r^{n-1}} \left( (n-3) f' - r f'' + ((n-2)f - rf') \left[ \Delta_z d(x,z) - \frac{n-1}{d(x,z)} \right] \right),
  $$
  where $r = d(x , z)$. By carefully choosing the smooth cut-off function $f$, the first two terms $|(n-3)f' - r f''|$ can be uniformly bounded by $C(n , \delta , i)$. With the help of Lemma \ref{Singularity Lemma}, we need to consider the integral
  $$
    \left| \int_{B_{\delta} (x_j)}  \frac{1}{d(x,z)^{n-1}} ((n-2)f - rf') \left[ \Delta_z d(x,z) - \frac{n-1}{d(x,z)} \right] \frac{1}{d(z,y)^{n-2}} dV_z \right|. 
  $$	
  First notice that $|(n-2)f - rf'| \leq C(n,\delta, i)$. Applying Bishop comparison theorem and G\"unther comparison theorem,
   $$
      (n-1) a \cot a d(x,z) \leq \Delta_ d(x,z) \leq (n-1) \sqrt{-\lambda} \coth \sqrt{-\lambda}  d(x,z).
   $$
   Using Taylor expansions for $\cot t$ and $\coth t$, there exists a uniform constant $C = C(n,\delta , i , a, \lambda)$ such that $|\Delta_z d(x,z) - \frac{n-1}{d(x,z)}| \leq C d(x,z)$. Therefore from Lemma \ref{Singularity Lemma}, $\Gamma_2 (x,y)$ is bounded by $d(x,y)^{n-3}$, up to some uniform constant. Then we argue inductively and hence conclude that $\Gamma_{k+1}$ together with its derivatives, are uniformly bounded by a uniform constant (for its derivatives, we leave the proof to readers). To obtain the uniform estimate for Green's function, it suffices to prove the uniform $C^1$ estimate of $F(x,y)$. 
	
	\begin{lem}
		Let $F(x,y)$ be a $W^{1,2}_0(B_\delta (x_j))$ solution of the Dirichlet problem \eqref{error term}, there exists a positive constant $C = C( n , \lambda , a,  i , Q )$ such that $|| \nabla_y F(x , y) ||_\infty \leq C$. 
	\end{lem}

	\begin{proof}
		Since $\Gamma_{k+1}$ in \eqref{error term} has no singularity, it is smooth and hence $F$ is the classical smooth solution. We first establish the $C^0$ estimate of $F(x, y )$, this follows from a classical barrier argument: consider the following Dirichlet problem on $B_\delta(x_j)$:
		
		\begin{equation}
			\begin{cases}
				- \Delta u = f(x),  \ \ \ & x\in B_\delta(x_j); \\
				u = 0, \ \ \ &x \in \partial B_\delta(x_j).
			\end{cases}
		\end{equation}
		We are to show: there exists a constant $C = C(n , Q, \delta)$ such that 
		$$
		||u||_\infty \leq C ||f||_\infty.
		$$ 
		In fact, let $w(x) = u(x) - z(x)$, where 
		$$
		z(x) = \frac{||f||_\infty}{A} (\delta^2 - |x|^2),
		$$
		$x$ is the position variable in terms of harmonic coordinates, and the constant $A$ is left to be chosen. It is easy to verify that 
		$$
		-\Delta w = f(x)- \frac{||f||_\infty}{A} \Delta |x|^2.
		$$
		Since with harmonic coordinates, we have
		
		$$
		\Delta |x|^2 = 2 |\nabla x|^2 + x \Delta x \in [2n Q^{-1} , 2n Q],
		$$
		then we can set $A = 2n Q^{-1}$ and hence $- \Delta w \leq 0$. Since $w \leq 0$ on $\partial B_\delta(x_j)$, by maximum principle, we have 
		$$
		u(x) \leq z(x) \leq \frac{\delta^2 Q ||f||_\infty}{2n}.
		$$
		Repeat the above argument for $- u(x)$ then we have
		
		$$
		-u (x) \geq - z(x) \geq - \frac{\delta^2 Q ||f||_\infty}{2n}.
		$$
		Therefore we obtain the $C^0$ estimate of $F$. Now we proceed the $C^{1 , \alpha}$ Schauder estimate on $B_\delta(x_j)$ so we obtain the $L^\infty$ estimate of $\nabla_yF(x,y)$. The proof reduces the problem to Euclidean boundary $C^{1, \alpha}$ estimates via harmonic coordinates and a flattening of the boundary. The argument is classical and standard, we only sketch the proof.
		
		\noindent\textbf{Step 1: Reduction to a Euclidean domain.}  
 		Let $\Phi:B_r(p)\to \Omega\subset\mathbb{R}^n$ be the harmonic coordinate chart.  
 		Let $\tilde u = u\circ \Phi^{-1}$, $\tilde f=f\circ \Phi^{-1}$, and $\tilde \varphi = \varphi\circ \Phi^{-1}$.  
 		Then $\tilde u$ solves a uniformly elliptic equation in divergence form on $\Omega$:
 		$$
 		-\partial_i(\tilde a^{ij}(x)\partial_j\tilde u)+\tilde b^i(x)\partial_i\tilde u+\tilde c(x)\tilde u
 		=\tilde f(x)\quad \text{in }\Omega,
 		$$
 		with boundary data $\tilde u=\tilde\varphi$ on $\partial\Omega$.
 		The coefficients satisfy
 		$$
 		\tilde a^{ij},\tilde b^i,\tilde c\in C^\alpha(\overline{\Omega}),
 		\qquad
 		\|\tilde a^{ij}\|_{C^\alpha}+\|\tilde b^i\|_{C^\alpha}+\|\tilde c\|_{C^\alpha}
 		\leq C(n,K,i_0,r).
 		$$
 		
 		\medskip
		
 		\noindent\textbf{Step 2: Flattening the boundary.}  
 		Since $r<i_0$, the boundary $\partial B_r(p)$ is smooth, and thus $\partial\Omega$ is $C^{1,\alpha}$.  
 		Cover $\partial\Omega$ by finitely many coordinate patches $\{U_k\}$ such that each patch is mapped by a $C^{1,\alpha}$-diffeomorphism $\Psi_k$ to a half-ball $B_1^+\subset\mathbb{R}^n$.  
 		Let $\{\eta_k\}$ be a partition of unity subordinate to this cover.
 		
 		\medskip

                \noindent\textbf{Step 3: Apply Euclidean boundary $C^{1,\alpha}$ estimates.}  
 		Fix one patch $U_k$ and define $v_k=\tilde u\circ \Psi_k^{-1}$ on $B_1^+$.  
 		Then $v_k$ solves a uniformly elliptic equation in divergence form with $C^\alpha$ coefficients and boundary data $v_k=\varphi_k$ on $\partial B_1^+$.  
 		By the Euclidean boundary $C^{1,\alpha}$ estimate (Gilbarg--Trudinger \cite{GilbargTrudinger}), we have
 		$$
 		\|v_k\|_{C^{1,\alpha}(\overline{B_{1/2}^+})}
 		\leq C\Big(
 		\|v_k\|_{L^\infty(B_1^+)}
 		+\|\tilde f\|_{C^\alpha(B_1^+)}
 		+\|\varphi_k\|_{C^{1,\alpha}(\partial B_1^+)}
 		\Big).
 		$$

		\noindent\textbf{Step 4: Patch local estimates to the whole domain.}  
 		Multiplying by $\eta_k$ and summing over $k$, we obtain a global estimate on $\Omega$:
 		$$
 		\|\tilde u\|_{C^{1,\alpha}(\overline{\Omega})}
 		\leq C\Big(
 		\|\tilde u\|_{L^\infty(\Omega)}
 		+\|\tilde f\|_{C^\alpha(\Omega)}
 		+\|\tilde\varphi\|_{C^{1,\alpha}(\partial\Omega)}
 		\Big).
 		$$
 		Using the uniform $C^{1,\alpha}$ control of the coordinate chart $\Phi$, we obtain the desired estimate.
		
	\end{proof}

\section{ Alternative proof of Trudinger-Moser inequalities using the uniform estimates of Green functions }

In this section, we give some remarks about how to prove
Theorem \ref{maintheorem} under additional condition of sectional curvature being upper bounded using the uniform Green function estimates. we achieve this through 2 steps: establishing the uniform Trudinger-Moser inequality on each $B_\delta (x_j)$; gluing all inequalities from Step 1 so we extend our inequality from the local inequalities to a  global one. We first prove
	
	\begin{thm}\label{local}
		Let $B_\delta (x_j)$ be any ball mentioned above, then there exists a constant $C = C(n , \lambda, i)$ such that 
		
		$$
		\sup_{u \in W^{1,n}_0 (B_\delta (x_j)) , ||\nabla u||_n \leq 1} \int_{B_\delta (x_j)} \phi_n (\alpha_n |u|^{\frac{n}{n-1}}) dV_g \leq C.
		$$
	\end{thm}
	 We use Adams' method to obtain an equivalent version of Theorem \ref{local}:
	
	\begin{thm}\label{equiv theorem}
		Let $B_\delta (x_j)$ be defined as previous and $T$ be the convolution operator such that 
		$$
		Tu (x) = \int_{B_\delta (x_j)} |\nabla G(x , y)| |\nabla u(y)| dV_g (y).
		$$
		Then there exists a constant $C = C( n , \lambda , i,  Q , \textit{Vol} (B_\delta(x_j)))$, which is increasing with respect to $\textit{Vol} (B_\delta(x_j))$, such that 
		
		$$
		\sup_{u \in L^n (B_\delta(x_j)) , ||u||_n \leq 1} \int_{B_\delta (x_j)} \phi_n \left( \frac{n}{\omega_{n-1}} |Tu|^{\frac{n}{n-1}} \right) dV_g \leq C.
		$$
	\end{thm}
	
	With the help of Green's function estimates in the previous section, the proof of Theorem \ref{equiv theorem} is exactly the same as that of Theorem 3.1 in \cite{LiLu-AIM}. The parameters $\lambda, i$ do not explicitly appear in the estimates, but will determine $Q$ and $r_H(\alpha , 0 , Q)$, which gives the  bounds of $\delta$.  This shows the constants in the Moser-Trudinger inequality on any geodesic balls only 
	depend on the radius $\delta$ which in turn depends on $\lambda$ and $i$.

 Proof of Theorem \ref{maintheorem} under sectional curvature upper bounded condition then follows from the uniform Green function estimates.

	\begin{rem}
	   There is another version of harmonic coordinates one can use to prove our main theorem. In \cite{AndersonCheeger}, it reads as a collection of finitely overlapping charts $\{ U_j , F_j \}$, where $F_j : U_j \to \mathbb{R}^n$ is of the form $U_j = F_j^{-1} (B_{r_H}(0))$. In such choice of harmonic coordinates, the a priori estimate of $F(x,y)$ will be significantly simplified since the Dirichlet problem is taken on Euclidean balls whose boundary is smooth and can be uniformly finitely covered. On the other hand, the domains $\{ U_j \}$ share the same Gromov's covering properties. Noticing that $\textit{Vol} (U_j) = \int_{U_j} dV = \int_{B_{r_H} (0)} \sqrt{|g|} dx \sim C(n , Q) r_H^n$, it is still possible to perform the local to global argument using this collection of charts, instead of geodesic balls.
	\end{rem}

\bibliographystyle{plain}

\begin{thebibliography}{10}

\bibitem{AndersonCheeger}
M.T. Anderson and J. Cheeger, $C^\alpha$-compactness for manifolds with Ricci curvature and
injectivity radius bounded below. J. Differential Geom. 35: 265-281, 1992.

\bibitem{Aubin}
T. Aubin, Nonlinear analysis on manifolds. Monge-Amp\'ere equations. Grundlehren der Mathematischen Wissenschaften, 252, Springer, New York-Berlin, 1982.

	\bibitem{ChenLuZhu-CVPDE2023}
	L. Chen, G. Lu and M. Zhu, 
	  Existence of extremals for Trudinger-Moser inequalities involved with a trapping potential. Calc. Var. Partial Differential Equations 62 (2023), no. 5, Paper No. 150, 35 pp.
	 
	\bibitem{ChenLuZhu-PLMS}
L.	Chen, G. Lu and M.  Zhu,  Least energy solutions to quasilinear subelliptic equations with constant and degenerate potentials on the Heisenberg group. Proc. Lond. Math. Soc. (3) 126 (2023), no. 2, 518–555.

 \bibitem{GilbargTrudinger}D. Gilbarg and N. Trudinger,
    Elliptic partial differential equations of second order. Reprint of the 1998 edition. Classics in Mathematics. Springer-Verlag, Berlin, 2001.
 
 \bibitem{Hebey} E. Hebey,
 Sobolev spaces on Riemannian manifolds,
 Lecture Notes in Mathematics, 1635. Springer-Verlag, Berlin, 1996. 

\bibitem{LamLu2}
N. Lam and G. Lu.
\newblock Sharp {M}oser-{T}rudinger inequality on the {H}eisenberg group at the
  critical case and applications.
\newblock {\em Adv. Math.}, 231(6):3259--3287, 2012.

\bibitem{LamLu1}
N. Lam and G. Lu.
\newblock A new approach to sharp {M}oser-{T}rudinger and {A}dams type
  inequalities: a rearrangement-free argument.
\newblock {\em J. Differential Equations}, 255(3):298--325, 2013.

\bibitem{LamLuTang} N. Lam, G. Lu and H. Tang, \newblock Sharp subcritical Moser-Trudinger inequalities on Heisenberg groups and subelliptic PDEs. \newblock {\em Nonlinear Anal.} 95 (2014), 77-92.
 
\bibitem{LamLuZhang} N. Lam, G. Lu and L. Zhang, Equivalence of critical and subcritical sharp Trudinger-Moser-Adams inequalities, Rev. Mat. Iberoam. 33 (4) (2017) 1219–1246.

\bibitem{LiLu-AIM}
 J. Li and G. Lu,  Critical and subcritical Trudinger-Moser inequalities on complete noncompact Riemannian manifolds. Adv. Math. 389 (2021), Paper No. 107915, 36 pp.
 
 
 \bibitem{LiLuZhu-ANS}
  J. Li, G. Lu and M. Zhu,  Concentration-compactness principle for Trudinger-Moser's inequalities on Riemannian manifolds and Heisenberg groups: a completely symmetrization-free argument. Adv. Nonlinear Stud. 21 (2021), no. 4, 917–937.
 
  
 \bibitem{LiLuZhu-CVPDE}
 J. Li, G. Lu and M. Zhu,  
   Concentration-compactness principle for Trudinger-Moser inequalities on Heisenberg groups and existence of ground state solutions. Calc. Var. Partial Differential Equations 57 (2018), no. 3, Paper No. 84, 26 pp. 
 
  
 \bibitem{LuTang-ANS1}
 G. Lu and H. Tang,   Best constants for Moser-Trudinger inequalities on high dimensional hyperbolic spaces. Adv. Nonlinear Stud. 13 (2013), no. 4, 1035–1052.
  
  \bibitem{LuTang-ANS2}
   G. Lu and H. Tang, Sharp singular Trudinger-Moser inequalities in Lorentz-Sobolev spaces. Adv. Nonlinear Stud. 16 (2016), no. 3, 581–601.
   
   \bibitem{YangQ} Q.Yang, D. Su and Y. Kong,
   Sharp Moser-Trudinger inequalities on Riemannian manifolds with negative curvature. Ann. Mat. Pura Appl. (4) 195 (2016), no. 2, 459–471.
 
\end{thebibliography}

\end{document}